\documentclass{article}
\usepackage{amsthm, amsmath, amssymb, mathrsfs}
\usepackage{graphicx}
\usepackage{tikz}
\usepackage{algpseudocode}
\usepackage{algorithm}
\usepackage{float}
\usepackage{natbib} % For author-year citations
\usepackage{hyperref}
\newtheorem{theorem}{Theorem}[section]
\newtheorem{lemma}[theorem]{Lemma}
\newtheorem{corollary}[theorem]{Corollary}
\newtheorem{proposition}[theorem]{Proposition}
\newtheorem{definition}[theorem]{Definition}

\title{\textbf{ISOMORPHISMS BETWEEN COVERING-INDUCED LATTICES \\ AND CLASSICAL GEOMETRIC LATTICES}}
\author{\textbf{Elvis Cabrera and Jyrko Correa}}

\date{} % Removes the date

\begin{document}

\maketitle

\section{Abstract}

Lattices induced by coverings arise naturally in matroid theory and combinatorial optimization, providing a structured framework for analyzing relationships between independent sets and closures. In this paper, we explore the structural properties of such lattices, with a particular focus on their rank structure, covering relations, and enumeration of elements per level. Leveraging these structural insights, we investigate necessary and sufficient conditions under which the lattice induced by a covering is isomorphic to classical geometric lattices, including the lattice of partitions, the lattice of subspaces of a vector space over a finite field, and the Dowling lattice. Our results provide a unified framework for comparing these combinatorial structures and contribute to the broader study of lattice theory, matroids, and their applications in combinatorics.

\section{Introduction}

   Lattices play a fundamental role in combinatorial mathematics, encoding hierarchical relationships between elements in a structured manner. Among these, geometric lattices have been extensively studied for their applications in combinatorics, topology, and algebraic geometry \citep{Bjorner1992, Oxley2006}. Notable examples include the lattice of partitions, the lattice of subspaces of a vector space over a finite field, and the Dowling lattice \citep{Dowling1973}, each of which provides a unique perspective on order-theoretic structures arising in mathematics. Understanding the isomorphisms between these lattices allows for deeper insights into their combinatorial properties and their connections to other mathematical disciplines.

In this work, we study the lattice induced by a covering, which naturally arises in matroid theory and combinatorial optimization \citep{Welsh1976, Oxley2006}. A key aspect of our study is the enumeration of elements per level, which is crucial for comparing the lattice induced by a covering with well-known lattices. We show that the number of elements per level follows an explicit combinatorial formula, allowing us to analyze whether its rank structure aligns with that of the lattice of partitions, the lattice of subspaces, or the Dowling lattice. Furthermore, we investigate the covering relation in the atoms to establish a foundation for determining isomorphisms between these lattices.

By identifying structural invariants and providing necessary and sufficient conditions for isomorphism, we aim to contribute to the broader understanding of lattice theory, matroids, and their applications in combinatorics. The results presented here offer a systematic approach to comparing different geometric lattices and provide a framework for further exploration of their combinatorial properties.

\section{Preliminaries}

In this section, we introduce key principles from matroid and lattice theory. Our focus includes the structural properties of three specific lattices: the partition lattice, the lattice of subspaces of a vector space over a finite field, and the Dowling lattice.

\subsection{The Matroid induced by a Covering}
Matroid theory generalizes ideas from linear algebra and graph theory, capturing core notions like independence, bases, and rank functions \citep{Oxley2006, Welsh1976}. Below, we formally define a matroid.

\begin{definition}[Matroid \citep{Whitney1935}]
A matroid is an ordered pair \( M = (U, \mathcal{J}) \), where \( U \) is a finite set and \( \mathcal{J} \) is a collection of subsets of \( U \) satisfying:
\begin{enumerate}
    \item[\( \mathbf{I1} \)] The empty set belongs to \( \mathcal{J} \), i.e., \( \emptyset \in \mathcal{J} \).
    \item[\( \mathbf{I2} \)] If \( I \in \mathcal{J} \) and \( I' \subseteq I \), then \( I' \in \mathcal{J} \).
    \item[\( \mathbf{I3} \)] For any \( I_1, I_2 \in \mathcal{J} \) where \( |I_1| < |I_2| \), there exists an element \( e \in I_2 \setminus I_1 \) such that \( I_1 \cup \{e\} \in \mathcal{J} \), where \( |X| \) denotes the cardinality of \( X \).
\end{enumerate}
\end{definition}

In a matroid \( M = (U, \mathcal{J}) \), the elements of \( \mathcal{J} \) are known as independent sets \citep{Oxley2006}.

A matroid is equipped with a rank function \( r_M : 2^U \to \mathbb{N} \), defined as:
$
    r_M(X) = \max\{|I| : I \subseteq X, I \in \mathcal{J}\}, \quad X \subseteq U.
$
The closure of a set \( X \) in \( M \) is given by:
$
    \mathrm{cl}_M(X) = \{a \in U : r_M(X) = r_M(X \cup \{a\})\}.
$
For convenience, we write \( \mathrm{cl}(X) \) instead of \( \mathrm{cl}_M(X) \). A set \( X \) is called a closed set of $M$ if it satisfies \( \mathrm{cl}(X) = X \) \citep{Welsh1976}.

One notable branch of matroid theory is transversal theory, which establishes connections between matroids and collections of subsets \citep{Lovasz1978}. Specifically, it introduces the concept of a transversal matroid, which is derived from a set system.

\begin{definition}[Transversal \citep{Mirsky1971}]
Let \( S \) be a finite set and \( J = \{1, 2, \dots, m\} \). Suppose \( \mathcal{F} = \{F_1, F_2, \dots, F_m\} \) is a family of subsets of \( S \). A transversal, also known as a system of distinct representatives, is a subset \( \{e_1, e_2, \dots, e_m\} \subseteq S \) such that \( e_i \in F_i \) for each \( i \in J \), and $e_i \neq e_j$ for $i\neq j$. If a subset \( X \) is a transversal for a subfamily \( \{F_i : i \in K\} \), where \( K \subseteq J \), then \( X \) is termed a partial transversal of \( \mathcal{F} \).
\end{definition}

\begin{proposition}[Transversal Matroid \citep{Mirsky1971}]
Given a family of subsets \( \mathcal{F} = \{F_i : i \in J\} \) of \( U \), the pair \( M(\mathcal{F}) = (U, \mathcal{J}(\mathcal{F})) \) defines a matroid, where \( \mathcal{J}(\mathcal{F}) \) consists of all partial transversals of \( \mathcal{F} \). This matroid is called the transversal matroid associated with \( \mathcal{F} \).
\end{proposition}

If $U$ is a set, we say $\mathcal{C} \subseteq 2^{U}$(power of $U$) if $\bigcup_{K\in C}K=U$. If $\mathcal{C}$ is a covering of $U$, we define $M(\mathcal{C})$ the transversal matroid of $\mathcal{C}$. Also, if $H \subseteq U$, let $\mathcal{C}(H):=\{C\in \mathcal{C}: C\cap H\neq \emptyset\}$ the collections of blocks of $\mathcal{C}$ that intersect $H$ non-trivially.
The following results characterize the collection of independent sets of this matroid. 
\begin{theorem}\label{matching1}\citep{Mirsky1971}
    If $\mathcal{C}$ is a covering of $U$ and $M(\mathcal{C})=(U,\mathcal{J})$ is the matroid of $\mathcal{C}$, then $H\subseteq U$ and $H\in \mathcal{J}$ if and only if $|\mathcal{C}(H')|\geq |H'|$, for all $H'\subseteq H$.
\end{theorem}

\begin{theorem}\label{matching2}\citep{Mirsky1971}
    Let $H \in \mathcal{J}$, $g \in U$, and $g \not\in H$. Then $g \in \operatorname{cl}(H)$ if and only if there exists $H' \subseteq H$ such that $|H'| = |\mathcal{C}(H')|$ and $\mathcal{C}(\{g\}) \subseteq \mathcal{C}(H')$.
    \end{theorem}

\subsection{Lattices}
Let \( (P, \leq) \) be a partially ordered set, and let \( a, b \in P \). We say that \( a \) is covered by \( b \) (or equivalently, \( b \) covers \( a \)) if \( a < b \) and no element \( c \in P \) satisfies \( a < c < b \) \citep{Birkhoff1940}.

A \textit{chain} in \( P \) from \( x_0 \) to \( x_n \) is a subset \( \{x_0, x_1, \dots, x_n\} \subseteq P \) such that \( x_0 < x_1 < \cdots < x_n \). The length of this chain is defined as \( n \), and the chain is considered maximal if each element \( x_i \) directly covers \( x_{i-1} \) for all \( i \in \{1, 2, \dots, n\} \) \citep{Stanley2012}. A poset \( P \) satisfies the \textit{Jordan-Dedekind chain condition} if, for every pair \( a, b \in P \) with \( a < b \), all maximal chains between \( a \) and \( b \) have the same length \citep{Gratzer2011}.

The \textit{height} \( h_P(y) \) of an element \( y \in P \) is the maximum length of a chain extending from the minimal element \( 0 \) to \( y \) \citep{Stanley2015}. A poset \( (L, \leq) \) qualifies as a lattice if every pair of elements \( a, b \in L \) has both a least upper bound (join \( a \vee b \)) and a greatest lower bound (meet \( a \wedge b \)).

If $0$ is the minimum element of a lattice \( L \), then any element \( a \in L \) that directly covers \( 0 \) is referred to as an \textit{atom} of \( L \) \citep{Birkhoff1940}. The atoms of \( L \) correspond precisely to the elements of height one. 
Lattices naturally partition into levels by grouping elements of equal height, where \( N(L) \) denotes the number of levels in \( L \).

\subsection{Closed-Set Lattice of a Matroid}
For a given matroid \( M \), let \( L(M) \) denote the collection of all closed sets of \( M \), ordered by inclusion. Then, \( (L(M), \subseteq) \) forms a lattice \citep{Oxley2006}. The meet and join operations within this structure are defined as:
$
X \wedge Y = X \cap Y, \quad X \vee Y = \mathrm{cl}_M(X \cup Y), \quad \forall X, Y \in L(M).
$
The smallest element of \( L(M) \) is \( \mathrm{cl}_M(\emptyset) \), while the greatest element is the ground set \( U \) of the matroid.

\subsection{Lattice Induced by a Covering}

The lattice induced by a covering, denoted \( L(\mathcal{C}) \), is the lattice of all closed sets in the transversal matroid associated with a covering \( \mathcal{C} \) of a finite set \citep{Mirsky1971, Lovasz1978}. Given a covering \( \mathcal{C} \), we define \( M(\mathcal{C}) \) as the transversal matroid generated by \( \mathcal{C} \). The elements of \( L(\mathcal{C}) \) are precisely the closed sets of \( M(\mathcal{C}) \), arranged in a hierarchy based on inclusion.

The meet of two closed sets in $L(\mathcal{C})$ is given by their intersection, representing the largest closed set contained within both. Similarly, the join operation is defined as the smallest closed set that contains their union \citep{Gratzer2011}. These operations ensure that \( L(\mathcal{C}) \) is a bounded, ranked lattice.

This lattice possesses a minimum element corresponding to the empty set and a maximum element equal to the union of all sets in \( \mathcal{C} \) \citep{Bjorner1992}. The elements of \( L(\mathcal{C}) \) can be further organized into levels according to height. Let \( L^{(\ell)}(\mathcal{C}) \) represent the set of closed sets at level \( \ell \), and let $|\operatorname{Cov}(X)|$ be the number of ways to extend  $X\in L^{(\ell)}(\mathcal{C})$ to a larger closed set at level $\ell+1$.

The lattice associated with a covering appears naturally in matroid theory, lattice theory, and combinatorial optimization \citep{Welsh1976}. In the following sections, we will investigate its structural properties and examine conditions under which it exhibits isomorphisms with other lattice structures.

\subsection{The Lattice of Partitions}

The lattice of partitions, denoted \( P_n \), consists of all partitions of an \( n \)-element set, ordered by refinement \citep{Stanley2012}. A partition of a set is a collection of nonempty, disjoint subsets, called blocks, whose union is the entire set. The refinement relation defines the partial order: given two partitions \( \pi \) and \( \sigma \), we say that \( \pi \leq \sigma \) if every block of \( \pi \) is contained within a block of \( \sigma \) \citep{Gratzer2011}.

This lattice has a minimum element, the discrete partition, where every element forms its own block, and a maximum element, the trivial partition, where all elements belong to a single block \citep{Birkhoff1940}. The elements of \( P_n \) can be grouped into levels based on their rank. We denote by \( P_n^{(\ell)} \) the set of partitions at the \( \ell \)-th level, where the rank function is given by  

$
\operatorname{rk}(\pi) = n - k,
$

with \( k \) being the number of blocks in \( \pi \) \citep{Stanley2012}.

The number of elements at each level is determined by the Stirling numbers of the second kind, denoted \( S(n, k) \), which count the ways to partition an \( n \)-element set into \( k \) blocks \citep{Comtet1974}. In particular, the first and second levels of \( P_n \) contain  

$
|P_n^{(1)}| = S(n, n-1) = \binom{n}{2}
$

$
|P_n^{(2)}| = S(n, n-2) = \frac{1}{24} n(n-1)(n-2)(3n-5).
$

These formulas will be fundamental in analyzing the structure of \( P_n \) and its potential isomorphisms with other lattices.

Each atom in \( P_n \) corresponds to a partition obtained by merging exactly two elements of the discrete partition. Thus, the number of atoms is precisely \( |P_n^{(1)}| = S(n, n-1) \) \citep{Stanley2012}. Furthermore, for each partition \( \pi \in P_n \), we define \( \operatorname{Cov}(\pi) \) as the set of partitions that cover \( \pi \), meaning those that are directly above \( \pi \) in the lattice. The number of elements covering an atom \( \pi \) is given by  

$
|\operatorname{Cov}(\pi)| = \binom{n-1}{2},
$

which corresponds to the number of ways to further merge two blocks. \citep{Comtet1974}.

The meet of two partitions is the coarsest partition that refines both, found by intersecting their blocks. The join is the finest partition coarser than both, obtained by taking the transitive closure of their union \citep{Gratzer2011}. These operations make \( P_n \) a bounded, ranked lattice.

\subsection{The Lattice of Subspaces of a Vector Space over a Finite Field}

The lattice of subspaces of an \( n \)-dimensional vector space \( V \) over a finite field \( \mathbb{F}_q \), denoted \( L(V) \), consists of all subspaces of \( V \), ordered by inclusion \citep{Stanley2012}. Each element of \( L(V) \) corresponds to a subspace of \( V \), and the partial order is given by set containment: for two subspaces \( U, W \in L(V) \), we say that \( U \leq W \) if \( U \subseteq W \) \citep{Birkhoff1940}.

This lattice has a minimum element, the zero subspace \( \{0\} \), and a maximum element, the entire space \( V \). The elements of \( L(V) \) can be grouped into levels based on their dimension. We denote by \( L^{(\ell)}(V) \) the set of subspaces of \( V \) with dimension \( \ell \), where the rank function is given by  

$
\operatorname{rk}(W) = \dim(W).
$

The number of elements at each level is determined by the Gaussian binomial coefficients, which count the number of \( \ell \)-dimensional subspaces of \( V \) \citep{Stanley2012, Kung1986}. In particular, the first and second levels of \( L(V) \) contain  

$
|L^{(1)}(V)| = \frac{q^n - 1}{q - 1}
$

$
|L^{(2)}(V)| = \frac{(q^n - 1)(q^{n-1} - 1)}{(q^2 - 1)(q-1)}.
$

These formulas will be essential in analyzing the structure of \( L(V) \) and its potential isomorphisms with other lattices \citep{Gratzer2011}.

Each atom in \( L(V) \) corresponds to a one-dimensional subspace of \( V \). Thus, the number of atoms is precisely \( |L^{(1)}(V)| \) \citep{Kung1986}. Furthermore, for each subspace \( W \in L(V) \), we define \( \operatorname{Cov}(W) \) as the set of subspaces that cover \( W \), meaning those that are directly above \( W \) in the lattice. The number of elements covering a one-dimensional subspace \( W \) is given by  

$
|\operatorname{Cov}(W)| = \frac{q^{n-1} - 1}{q - 1},
$

which corresponds to the number of ways to extend \( W \) to a two-dimensional subspace \citep{Stanley2012}.

The meet of two subspaces is their intersection, which gives the largest subspace contained in both. The join is their sum, which gives the smallest subspace containing both. These operations make \( L(V) \) a bounded, ranked lattice \citep{Birkhoff1940}.

\subsection{The Dowling Lattice}

Let \(G\) be a finite group and let \(N=\{1,\dots,n\}\).
The \emph{Dowling lattice} \(Q_n(G)\) is the geometric lattice associated to the
Dowling geometry of rank \(n\) \citep{Dowling1973}.  A convenient concrete model
is in terms of \emph{partial \(G\)-partitions} with a distinguished \emph{zero block}.

\begin{definition}[\(G\)-labeled blocks and partial \(G\)-partitions]
A \emph{\(G\)-labeled block} is a pair \((B,a)\) where \(B\subseteq N\) is nonempty
and \(a:B\to G\) is a function. Two labelings \(a,a':B\to G\) are called
\emph{equivalent} (and we write \((B,a)\sim(B,a')\)) if there exists \(g\in G\) such that
\[
a'(x)=g\,a(x)\qquad \text{for all }x\in B,
\]
i.e., they differ by \emph{left multiplication by a constant} on the block.

A \emph{partial \(G\)-partition} of \(N\) is data
\[
\Pi=\bigl(Z;\ (B_1,a_1),\dots,(B_k,a_k)\bigr)
\]
where \(Z\subseteq N\) (the \emph{zero block}), the sets \(B_1,\dots,B_k\) form a partition
of \(N\setminus Z\) into nonempty blocks, and each \(a_i:B_i\to G\) is a labeling.
Two partial \(G\)-partitions are identified if they have the same zero block and,
for each \(i\), the labeled blocks \((B_i,a_i)\) agree up to the equivalence above
(independently on each nonzero block).
An element of \(Q_n(G)\) is an equivalence class \([\Pi]\) of such data.
\end{definition}

\begin{definition}[Order and rank]
Let \(x=[(Z;\ (B_1,a_1),\dots,(B_k,a_k))]\) and
\(y=[(Z';\ (C_1,b_1),\dots,(C_r,b_r))]\).
We define \(x\le y\) if:
\begin{enumerate}
\item \(Z\subseteq Z'\); and
\item for each nonzero block \(C_j\) of \(y\), there exists a (set) partition of \(C_j\)
into blocks among \(\{B_1,\dots,B_k\}\), and for every \(B_i\subseteq C_j\) there exists
\(g_{i}\in G\) such that
\[
b_j(x)=g_i\,a_i(x)\qquad \text{for all }x\in B_i.
\]
(Equivalently, the labeling on each refined block agrees with the labeling of \(C_j\)
up to a blockwise left multiplication.)
\end{enumerate}
With this order, \(Q_n(G)\) is a lattice of rank \(n\). If \(x\) has \(k\) nonzero blocks,
its rank is
\[
\operatorname{rk}(x)=n-k.
\]
In particular, the minimum element is the discrete \(G\)-partition
\((\varnothing;\ (\{1\},1),\dots,(\{n\},1))\) (so \(k=n\)), and the maximum element is
\((N;\ )\) (all elements in the zero block, so \(k=0\)).
\end{definition}

We denote by \(Q_n^{(\ell)}(G)\) the set of elements of rank \(\ell\) in \(Q_n(G)\).
The cardinalities \(|Q_n^{(\ell)}(G)|\) are the Whitney numbers of the second kind of
\(Q_n(G)\). Writing \(m=|G|\), one has the standard formula \citep{Dowling1973}
\[
|Q_n^{(\ell)}(G)|
=
W_m(n,n-\ell)
=
\sum_{i=n-\ell}^{n}\binom{n}{i}\,m^{\,i-(n-\ell)}\,S(i,n-\ell),
\]
where \(S(i,k)\) denotes the Stirling numbers of the second kind.

For the first two levels this gives:
\[
|Q_n^{(1)}(G)| = n + \binom{n}{2}\,|G|,
\]
\[
|Q_n^{(2)}(G)|
= \binom{n}{2}
+ \frac12\,n(n-1)(n-2)\,|G|
+ \frac{1}{24}\,n(n-1)(n-2)(3n-5)\,|G|^2.
\]

The atoms (rank \(1\) elements) come in two natural types:
\begin{itemize}
\item \emph{Zero-block atoms:} move one element \(i\) into the zero block (so the zero block is \(\{i\}\) and all other elements remain singleton nonzero blocks).
\item \emph{Pair atoms:} merge two singleton blocks \(\{i\},\{j\}\) into one nonzero block \(\{i,j\}\) with a relative label in \(G\).
\end{itemize}
Hence \(|Q_n^{(1)}(G)| = n + \binom{n}{2}|G|\) as above.

\medskip
Finally, let \(x\in Q_n(G)\) have \(k\) nonzero blocks (so \(\operatorname{rk}(x)=n-k\)).
Then the number of elements that cover \(x\) depends only on \(k\) (equivalently, only on
\(\operatorname{rk}(x)\)), and is given by
\[
|\operatorname{Cov}(x)| \;=\; k \;+\; |G|\binom{k}{2}.
\]

\section{Isomorphisms Between Lattices Induced by Coverings and Other Geometric Lattices}

We begin by characterizing isomorphism classes of lattices induced by coverings. Establishing these classes simplifies their structural analysis and facilitates direct comparisons with other lattice families.

If \(\mathcal{C}\) is a covering of a finite set \(U\), define the equivalence relation $\sim$ in $U$ by $x \sim y$ if and only if $x=y$, or $\mathcal{C}(\{x\}) = \mathcal{C}(\{y\})$ and $|\mathcal{C}(\{x\})| = 1$. Then, it is natural to set $[\mathcal{C}]_{\sim} := \{K/{\sim}: K \in \mathcal{C}\}$. Note that $[\mathcal{C}]_{\sim}$ is a covering of $U/{\sim}$. The following theorem explores the structural relationships between $L(\mathcal{C}) \text{ and } L([\mathcal{C}]_{\sim})$. 

\begin{theorem}\label{simplification}
$
L(\mathcal{C}) \cong L([\mathcal{C}]_{\sim})
$
\end{theorem}

\begin{proof}

\noindent
For every \( X \in L(\mathcal{C}) \), define the function
$
f : L(\mathcal{C}) \to L([\mathcal{C}]_{\sim})
\text{      by       }
f(X) = X/{\sim}.
$
We first show that \( f \) is well-defined; that is, if \( X \in L(\mathcal{C}) \), then \( X/{\sim} \in L([\mathcal{C}]_{\sim}) \).

Since \( X \in L(\mathcal{C}) \), for every \( u \in U \), there exists \( J_u \in \mathcal{J}_{M(\mathcal{C})} \) such that
$\operatorname{rank}_{M(\mathcal{C})}(X) = \operatorname{rank}_{M(\mathcal{C})}(J_u)$
and
$\operatorname{rank}_{M(\mathcal{C})}(J_u \cup \{u\}) = \operatorname{rank}_{M(\mathcal{C})}(J_u) + 1.$

For every equivalence class \( [u]_{\sim} \in U/{\sim} \), observe that:
\begin{itemize}
    \item \( J_u/{\sim} \in \mathcal{J}_{M([\mathcal{C}]_{\sim})} \),
    \item \( J_u/{\sim} \subseteq X/{\sim} \),
    \item \( \operatorname{rank}_{M([\mathcal{C}]_{\sim})}(X/{\sim}) = \operatorname{rank}_{M([\mathcal{C}]_{\sim})}(J_u/{\sim}) \),
    \item \( \operatorname{rank}_{M([\mathcal{C}]_{\sim})}(J_u/{\sim} \cup [u]_{\sim}) = \operatorname{rank}_{M([\mathcal{C}]_{\sim})}(J_u/{\sim}) + 1 \).
\end{itemize}
Hence, \( X/{\sim} \in L([\mathcal{C}]_{\sim}) \).

The map \( f \) is order-preserving because for any \( X, Y \in L(\mathcal{C}) \) with \( X \subseteq Y \), it follows that \( X/{\sim} \subseteq Y/{\sim} \).

Moreover, \( f \) is injective. Indeed, observe that if
$x\sim y$ and $x \in X\in L(\mathcal{C})$, then $y \in X$. Thus, if \( X/{\sim} = Y/{\sim} \), then \( X = Y \).

Finally, to show that \( f \) is surjective, let \( X/{\sim} \in L([\mathcal{C}]_{\sim}) \). We must prove that \( X \in L(\mathcal{C}) \). In other words, we want to show that if \( X/{\sim} \) is a flat of \( M([\mathcal{C}]_{\sim}) \), then \( X \) is a flat of \( M(\mathcal{C}) \). This can be established using an argument analogous to the one given for well-definedness.

\end{proof}

Thus, without loss of generality, we assume throughout that all coverings satisfy that the equivalence classes given by $\sim$ are singletons, i.e. \( L^{(1)}(\mathcal{C}) = \{\{x\} : x \in U\} \). The preceding theorem ensures that this assumption does not impose any restrictions, as every covering can be transformed into an isomorphic one satisfying this condition.

    \subsection{Isomorphism with a Uniform Lattice}

In this section, we examine the structural properties of a lattice induced by a covering that is isomorphic to a lattice \(\mathcal{L}\) satisfying the condition:
\[
\text{If } x, y \in \mathcal{L} \text{ and } \operatorname{rk}(x) = \operatorname{rk}(y), \text{ then } |\operatorname{Cov}(x)| = |\operatorname{Cov}(y)|. \tag{uniformity}
\]

Let \( L(\mathcal{C}) \cong \mathcal{L} \) for some covering \( \mathcal{C} \) with \( |U| = m \). Since \( L^{(1)}(\mathcal{C}) = \{\{x\} : x \in U\} \), it follows that $
m = |L^{(1)}(\mathcal{C})| = |\mathcal{L}^{(1)}|.
$

Additionally, we have 
$
|L^{(2)}(\mathcal{C})| \leq \binom{m}{2},
$
with equality holding if and only if 
$
\operatorname{cl}(\{a,b\}) = \{a,b\}, \quad \text{for all } a, b \in U.
$
In this case, \( |\mathcal{L}^{(2)}| = \binom{m}{2} \).

We now analyze the case where \( |L^{(2)}(\mathcal{C})| < \binom{m}{2} \), which occurs precisely when there exist \( a, b, c \in U \) such that 
$
b \in \operatorname{cl}(\{a,c\}).
$
The following lemma provides necessary conditions for this scenario.

\begin{lemma}
If there exist distinct \( a, b, c \in U \) such that \( b \in \operatorname{cl}(\{a,c\}) \), then at least one of them satisfies \( |\mathcal{C}(\{x\})| = 2 \).
\end{lemma}

\begin{proof}
It is straightforward to see that none of them satisfies \( |\mathcal{C}(\{x\})| \geq 3 \). Suppose, for contradiction, that \( |\mathcal{C}(\{a\})| = |\mathcal{C}(\{b\})| = |\mathcal{C}(\{c\})| = 1 \). If \( \{a,b,c\} \not\in \mathcal{J} \), then by Theorem 4.1, there exists a subset \( H \subseteq \{a,b,c\} \) such that \( |\mathcal{C}(H)| < |H| \). 

    Since \( \{a,b\}, \{a,c\}, \{b,c\} \in \mathcal{J} \), it follows that \( H = \{a,b,c\} \), implying that \( |\mathcal{C}(\{a,b,c\})| < 3 \). This means that two of these elements belong to the same block of \( \mathcal{C} \), which is a contradiction.
\end{proof}

In our argument, we will compare the covers of the atoms of both lattices. The following lemma provides insight into the cover structure of the atoms of \( L(\mathcal{C}) \).

\begin{lemma}
Let \( x,a,b \) be different elements of $U$ such that \( |\mathcal{C}(\{x\})| = 2 \). Then, \( \{x,a,b\} \not\in \mathcal{J} \) if and only if \( \mathcal{C}(\{a\}) \subseteq \mathcal{C}(\{x\}) \) and \( \mathcal{C}(\{b\}) \subseteq \mathcal{C}(\{x\}) \).
\end{lemma}

\begin{proof}
 Suppose that \( \{x,a,b\} \not\in \mathcal{J} \). By condition (1), we know that \( \{x,a\} \in \mathcal{J} \). Then, by Theorem 4.2, we have \( \mathcal{C}(\{b\}) \subseteq \mathcal{C}(\{x,a\}) \) and \( |\mathcal{C}(\{x,a\})| = 2 \). Since \( |\mathcal{C}(\{x\})| = 2 \), it follows that \( \mathcal{C}(\{a\}) \subseteq \mathcal{C}(\{x\}) \), and consequently, \( \mathcal{C}(\{b\}) \subseteq \mathcal{C}(\{x\}) \).

 Suppose that \( \mathcal{C}(\{a\}) \subseteq \mathcal{C}(\{x\}) \) and \( \mathcal{C}(\{b\}) \subseteq \mathcal{C}(\{x\}) \). Then, since \( |\{x,a,b\}| = 3 \) and \( |\mathcal{C}(\{x,a,b\})| = 2 \), Theorem~\ref{matching1} ensures that \( \{x,a,b\} \not\in \mathcal{J} \).
\end{proof}

As a direct consequence, we obtain the following corollary.

\begin{corollary}
If \( |\mathcal{C}(\{x\})| = 2 \) and \( \mathcal{C}(\{a\}) \not\subseteq \mathcal{C}(\{x\}) \), then \( \operatorname{cl}(\{x,a\}) = \{x,a\} \).
\end{corollary}

\begin{proof}
Suppose, for contradiction, that there exists \( b \not\in \{x,a\} \) such that \( b \in \operatorname{cl}(\{x,a\}) \). By Theorem 4.2, this would imply that \( |\mathcal{C}(\{x,a\})| = 2 \). However, the condition \( \mathcal{C}(\{a\}) \not\subseteq \mathcal{C}(\{x\}) \) guarantees that \( |\mathcal{C}(\{x,a\})| > 2 \), leading to a contradiction.
\end{proof}

For \( x \in U \) satisfying \( |\mathcal{C}(\{x\})| = 2 \), define
$E(\{x\}) := \{y \in U : \mathcal{C}(\{y\}) \subseteq \mathcal{C}(\{x\})\}.$
Observe that \( \operatorname{cl}(\{a,b\}) = E(\{x\}) \) for every \( a,b \in E(\{x\}) \) and $a\neq b$. Furthermore, \( E(\{x\}) = E(\{y\}) \) if and only if \( \mathcal{C}(\{x\}) = \mathcal{C}(\{y\})\) and $|\mathcal{C}(\{x\})|=2$ .

Additionally, if \( \pi \in \mathcal{L}^{(1)} \), then \( |\operatorname{Cov}(\pi)| = |\operatorname{Cov}(\{u\})| \) for every \( u \in U \). We now proceed to compute \( |E(\{x\})| \).

    \begin{theorem}
    If \( x \) satisfies \( |\mathcal{C}(\{x\})| = 2 \), then 
    $|E(\{x\})| = |\mathcal{L}^{(1)}| - |\operatorname{Cov}(\pi)| + 1.$
    
\end{theorem}

\begin{proof}
    From the preceding lemma, we observe that the elements of \( L(\mathcal{C}) \) that cover \( \{x\} \) are \( E(\{x\}) \) and sets of the form \( \{x,a\} \) where \( \mathcal{C}(\{a\}) \not\subseteq \mathcal{C}(\{x\}) \). The number of such sets is \( m - |E(\{x\})| \). Thus, we obtain 
$    |\operatorname{Cov}(\{x\})| = m - |E(\{x\})| + 1.$
  Furthermore, since it is known that \( m = |\mathcal{L}^{(1)}| \) and \( |\operatorname{Cov}(\{x\})| = |\operatorname{Cov}(\pi)| \), it follows that
    $|E(\{x\})| = |\mathcal{L}^{(1)}| - |\operatorname{Cov}(\pi)| + 1.$

\end{proof}

\begin{corollary}
For all \( x, y \) such that \( |\mathcal{C}(\{x\})| = |\mathcal{C}(\{y\})| = 2 \), we have
    $
    |E(\{x\})| = |E(\{y\})| = n.
    $
\end{corollary}

The following two lemmas examine the cases where \( |\mathcal{C}(\{a\})| = 1 \) or \( |\mathcal{C}(\{a\})| \geq 3 \).

\begin{lemma}
    Let \( \mathcal{E} := \{E(\{x\}) : |\mathcal{C}(\{x\})|=2\} \). If \( a \in U \) satisfies \( |\mathcal{C}(\{a\})| = 1 \), then \( a \) belongs to exactly one element of \( \mathcal{E} \).
\end{lemma}

\begin{proof}
    Suppose, for contradiction, that there exist at least two elements in \( \mathcal{E} \), say \( E_1 \) and \( E_2 \), such that \( a \in E_1 \) and \( a \in E_2 \). Then, we have
    $
    |\operatorname{Cov}(\{a\})| \leq m - |E_1| - |E_2| + 3.
    $
    Since \( |\operatorname{Cov}(\{a\})| = |\operatorname{Cov}(\{x\})| = m - |E(\{x\})| + 1 \), it follows that
    $
    m - |E(\{x\})| + 1 \leq m - |E_1| - |E_2| + 3,
    $
    which is a contradiction.

    Now, suppose there exists \( a \) such that \( |\mathcal{C}(\{a\})| = 1 \) and \( a \) does not belong to any element of \( \mathcal{E} \). Then, we would have \( |\operatorname{Cov}(\{a\})| = m - 1 \), which is also a contradiction.
\end{proof}

\begin{lemma}
    There is no \( a \in U \) such that \( |\mathcal{C}(\{a\})| \geq 3 \).
\end{lemma}

\begin{proof}
    If \( |\mathcal{C}(\{a\})| \geq 3 \), then \( |\operatorname{Cov}(\{a\})| = m - 1 \), which contradicts our assumptions.
\end{proof}

From the above lemmas, it follows that \( \mathcal{E} \) forms a partition of \( U \). Since each element of \( \mathcal{E} \) has cardinality \( |E(\{x\})| \), we obtain
$
m = k \cdot |E(\{x\})|
$, for some natural number $k$.
Thus, it follows directly that
$
k = \frac{|\mathcal{L}^{(1)}|}{|\mathcal{L}^{(1)}| - |\operatorname{Cov}(\pi)| + 1}.
$

The following theorem determines the number of elements in \( L^{(2)}(\mathcal{C}) \).

\begin{theorem}\label{counting}
    If $\pi$ is an atom of $\mathcal{L}$, then
    \[
    |\mathcal{L}^{(2)}| = |L^{(2)}(\mathcal{C})| = \binom{|\mathcal{L}^{(1)}|}{2} - \left(\binom{|\mathcal{L}^{(1)}| - |\operatorname{Cov}(\pi)| + 1}{2} - 1\right) \cdot \frac{|\mathcal{L}^{(1)}|}{|\mathcal{L}^{(1)}| - |\operatorname{Cov}(\pi)| + 1}.
    \]
\end{theorem}

\begin{proof}
    It is straightforward to see that if \( a, b \) do not belong to the same element of \( \mathcal{E} \), then \( \operatorname{cl}(\{a,b\}) = \{a,b\} \). Conversely, if \( a, b \) belong to the same element of \( \mathcal{E} \), say \( \hat{E} \), then \( \operatorname{cl}(\{a,b\}) = \hat{E} \).

    We begin with \( \binom{m}{2} \), which counts all subsets of \( U \) with two elements. Each of these subsets has a unique closure, except for those pairs \( \{a, b\} \) where \( a, b \) belong to the same block in \( \mathcal{E} \). 

    Since there are \( k \) blocks in \( \mathcal{E} \), each containing \( |E(\{x\})| \) elements, we subtract the redundancy introduced by these cases and add back the number of blocks in \( \mathcal{E} \). This yields the proposed formula.
\end{proof}

\subsection{Isomorphisms with the Lattice of Partitions}

In this section, we examine all possible cases of isomorphism between the lattice induced by a covering and the lattice of partitions.

\begin{theorem}\label{thm:Pn-necessary}
If \(P_n \cong L(\mathcal C)\), then \(n\le 3\).
Moreover, if \(n=3\), then \(m=3\) and \(k=1\).
\end{theorem}

\begin{proof}
Since lattice isomorphisms preserve rank, the number of atoms (rank--\(1\) elements) must agree:
\[
|P_n^{(1)}| = |L^{(1)}(\mathcal C)|.
\]
Now \(|P_n^{(1)}| = S(n,n-1)=\binom{n}{2}\), while \(|L^{(1)}(\mathcal C)|=m\). Hence
\begin{equation}\label{eq:m-binomial}
m=\binom{n}{2}=\frac{n(n-1)}{2}.
\end{equation}

\medskip
\noindent\textbf{Case 1: \(\;|L^{(2)}(\mathcal C)|=\binom{m}{2}\).}
Using \eqref{eq:m-binomial},
\[
\binom{m}{2}=\binom{\binom{n}{2}}{2}.
\]
For \(n\ge 3\), we also have
\[
|L^{(2)}(\mathcal C)|=|P_n^{(2)}|=S(n,n-2)=\frac{1}{24}n(n-1)(n-2)(3n-5).
\]
Therefore,
\begin{equation}\label{eq:case1-eq}
\binom{\binom{n}{2}}{2}=\frac{1}{24}n(n-1)(n-2)(3n-5).
\end{equation}
Expand the left-hand side carefully:
\[
\binom{\binom{n}{2}}{2}
=\frac{1}{2}\binom{n}{2}\!\left(\binom{n}{2}-1\right)
=\frac{1}{2}\cdot \frac{n(n-1)}{2}\left(\frac{n(n-1)}{2}-1\right).
\]
Rewrite the last factor with a common denominator:
\[
\frac{n(n-1)}{2}-1=\frac{n(n-1)-2}{2}.
\]
Hence,
\[
\binom{\binom{n}{2}}{2}
=\frac{1}{2}\cdot \frac{n(n-1)}{2}\cdot \frac{n(n-1)-2}{2}
=\frac{n(n-1)\bigl(n(n-1)-2\bigr)}{8}.
\]
Substitute into \eqref{eq:case1-eq} and clear denominators by multiplying by \(24\):
\[
24\cdot \frac{n(n-1)\bigl(n(n-1)-2\bigr)}{8}
= n(n-1)(n-2)(3n-5),
\]
so
\[
3\,n(n-1)\bigl(n(n-1)-2\bigr)=n(n-1)(n-2)(3n-5).
\]
For \(n\ge 3\), we have \(n(n-1)\neq 0\), so we cancel \(n(n-1)\):
\[
3\bigl(n(n-1)-2\bigr)=(n-2)(3n-5).
\]
Expand both sides:
\[
3(n^2-n-2) = 3n^2-11n+10.
\]
The left-hand side is \(3n^2-3n-6\), hence
\[
3n^2-3n-6 = 3n^2-11n+10
\quad\Longrightarrow\quad
8n=16
\quad\Longrightarrow\quad
n=2,
\]
which contradicts \(n\ge 3\). Therefore, for \(n\ge 3\) no isomorphism can occur in Case 1.

\medskip
\noindent\textbf{Case 2: \(\;|L^{(2)}(\mathcal C)|<\binom{m}{2}\).}
By Theorem~\ref{counting} and \eqref{eq:m-binomial},
\[
|L^{(2)}(\mathcal C)|
=\binom{m}{2}-m\cdot\frac{n-1}{2}+\frac{n-1}{2}
=\binom{\binom{n}{2}}{2}-\binom{n}{2}\cdot\frac{n-1}{2}+\frac{n-1}{2}.
\]
Equating \( |L^{(2)}(\mathcal C)|=|P_n^{(2)}| \) gives
\begin{equation}\label{eq:case2-master}
\binom{\binom{n}{2}}{2}-\binom{n}{2}\cdot\frac{n-1}{2}+\frac{n-1}{2}
=\frac{1}{24}n(n-1)(n-2)(3n-5).
\end{equation}
We simplify the left-hand side term-by-term.
First,
\[
\binom{\binom{n}{2}}{2}=\frac{n(n-1)\bigl(n(n-1)-2\bigr)}{8}.
\]
Second,
\[
\binom{n}{2}\cdot\frac{n-1}{2}
=\frac{n(n-1)}{2}\cdot\frac{n-1}{2}
=\frac{n(n-1)^2}{4}.
\]
Substitute these into \eqref{eq:case2-master}:
\[
\frac{n(n-1)\bigl(n(n-1)-2\bigr)}{8}-\frac{n(n-1)^2}{4}+\frac{n-1}{2}
=\frac{1}{24}n(n-1)(n-2)(3n-5).
\]
Clear denominators by multiplying by \(24\):
\[
3n(n-1)\bigl(n(n-1)-2\bigr)-6n(n-1)^2+12(n-1)
= n(n-1)(n-2)(3n-5).
\]
Factor out \((n-1)\) on the left:
\[
(n-1)\Bigl(3n\bigl(n(n-1)-2\bigr)-6n(n-1)+12\Bigr)
= n(n-1)(n-2)(3n-5).
\]
For \(n\ge 2\), \(n-1\neq 0\), so we cancel \((n-1)\):
\[
3n\bigl(n(n-1)-2\bigr)-6n(n-1)+12
= n(n-2)(3n-5).
\]
Now expand the left-hand side:
\[
3n\bigl(n^2-n-2\bigr)-6n(n-1)+12
= (3n^3-3n^2-6n) + (-6n^2+6n) + 12
= 3n^3-9n^2+12.
\]
So the equation becomes
\[
3n^3-9n^2+12 = n(n-2)(3n-5).
\]
Expand the right-hand side:
\[
n(n-2)(3n-5)=n(3n^2-11n+10)=3n^3-11n^2+10n.
\]
Equate and cancel \(3n^3\):
\[
-9n^2+12 = -11n^2+10n
\]
\[
\Longrightarrow\quad 2n^2-10n+12=0
\]
\[
\Longrightarrow\quad n^2-5n+6=0
\]
\[
\Longrightarrow\quad (n-2)(n-3)=0.
\]

Thus \(n\leq 3\). Specifically, if $n=3$, then \eqref{eq:m-binomial} gives \(m=\binom{3}{2}=3\), and in this case the covering must have \(k=1\).
\end{proof}

\begin{lemma}\label{heights0,1}

If $\mathcal{C}$ is a covering of a set $U$, then
\begin{enumerate}
    \item $|N(L(\mathcal{C}))|=1$ if and only if $U = \emptyset$.
    \item $|N(L(\mathcal{C}))|=2$ if and only if $\mathcal{C}=\{U\}$.
\end{enumerate}
    
\end{lemma}

\begin{proof}

    \begin{enumerate}
        \item First, suppose that $|N(L(\mathcal{C}))|=1$ (equivalently, $|L(\mathcal{C})|=1$). Thus, if $\mathfrak{m}\in L(\mathcal{C})$, then $\text{rank}_{L(\mathcal{C})}(m)=0$. This entails that if $M(U, \mathcal{J})$ is the matroid of $\mathcal{C}$, then for all $S\subseteq U$, we have $\text{rank}_M(S)=0$. Naturally, this implies $\mathcal{J}=\emptyset$, which is only possible if $U = \emptyset$.

        To prove the converse, suppose that $U = \emptyset$. This implies that $\mathcal{J}=\emptyset$, and consequently, the only closed set of $M$ is $U$. Hence, $|N(L(\mathcal{C}))|=1$
        \item First, suppose that $|N(L(\mathcal{C}))|=2$. Observe that $\mathcal{J}\neq \emptyset$, which entails that the minimum and maximum elements of $L(\mathcal{C})$ are $\emptyset$ and $U$, respectively. Thus, if $|N(L(\mathcal{C}))|=2$, then $\emptyset$ and $U$ are the only elements in $L(\mathcal{C})$. This means that if $a \in U$, then $\text{cl}_M(\{a\})=U$. Note that if $U=\{a\}$, it follows that $|N(L(\mathcal{C}))|=2$. Suppose that $b\in U\setminus\{a\}$. Since $b \in \text{cl}_M(\{a\})$, Theorem~\ref{matching2} guaranties that $|\mathcal{C}(\{a\})|=1$ and $ \mathcal{C}(\{b\})\subseteq \mathcal{C}(\{a\})$. Hence, $\mathcal{C} = \{U\}$.

        To show the converse, assume that $\mathcal{C} = \{U\}$. It follows that $\mathcal{J}= \{\{a\}\}_{a\in U}\cup\{\emptyset\}$. From the above, it is not hard to verify that $\emptyset$ and $U$ are the only elements in $L(\mathcal{C})$, which implies that $|N(L(\mathcal{C}))|=2$.
    \end{enumerate}
    
\end{proof}

\begin{theorem}\label{thm:Pn-sufficient}
    If $n=1$ or $n=2$, then the lattice of partitions is trivially isomorphic to the lattice induced by a covering. If $n=3$, $m=3$, and $k=1$, then $L(\mathcal{C})\cong P_3$.
\end{theorem}

\begin{proof}
     First, observe that all lattices of height  $0$ are isomorphic and all lattices of height  $1$  are isomorphic. Thus, if $n=1$, then $L(\mathcal{C})\cong P_1$ if and only if $N(L(\mathcal{C}))=1$, which by Lemma~\ref{heights0,1}, is equivalent to $U = \emptyset$. Similarly, if $n=2$, then $L(\mathcal{C})\cong P_2$ if and only if $\mathcal{C}=\{U\}$.
     
     Now, if $n=3$, $|L^{(2)}(\mathcal{C})| < \binom{m}{2}$, $m=3$, and $k=1$, then \( |N(P_n)| = |N(L(\mathcal{C}))| = 3 \). Hence, it is sufficient to show that these lattices have the same number of atoms. Indeed, since \( m=3 \), it follows that \( |L^{(1)}(\mathcal{C})|=|P_n^{(1)}|=3 \), which completes the proof.
\end{proof}

Combining Theorems~\ref{thm:Pn-necessary} and \ref{thm:Pn-sufficient}, we obtain a complete classification of when $L(\mathcal{C})$ can be isomorphic to a partition lattice $P_n$.

Next, we illustrate the isomorphism between the lattice induced by the covering \( \mathcal{C} = \{\{a,b\},\{a,c\}\} \) and \( P_3 \),  using Hasse diagrams.

\begin{center}
\begin{tabular}{@{}c@{\hspace{1.4cm}}c@{}}

% ------------ Left: P3 ------------
\begin{minipage}{0.45\textwidth}\centering
\resizebox{\linewidth}{!}{%
\begin{tikzpicture}[every node/.style={font=\small}]
  % Force identical bounding box (ignore label widths)
  \path[use as bounding box] (-2,-0.6) rectangle (2,4.2);

  \node at (0,3.8) {\textbf{Lattice of Partitions, \(P_3\)}};

  \node (top) at (0,3) {abc};
  \node (L)   at (-1.5,1.5) {a$|$bc};
  \node (M)   at (0,1.5) {ab$|$c};
  \node (R)   at (1.5,1.5) {ac$|$b};
  \node (bot) at (0,0) {a$|$b$|$c};

  \draw (top)--(L) (top)--(M) (top)--(R);
  \draw (L)--(bot) (M)--(bot) (R)--(bot);
\end{tikzpicture}%
}
\end{minipage}
&
% ------------ Right: L(C) ------------
\begin{minipage}{0.45\textwidth}\centering
\resizebox{\linewidth}{!}{%
\begin{tikzpicture}[every node/.style={font=\small}]
  % SAME bounding box as left
  \path[use as bounding box] (-2,-0.6) rectangle (2,4.2);

  \node at (0,3.8) {\textbf{Lattice Induced by Covering, \(L(\mathcal{C})\)}};

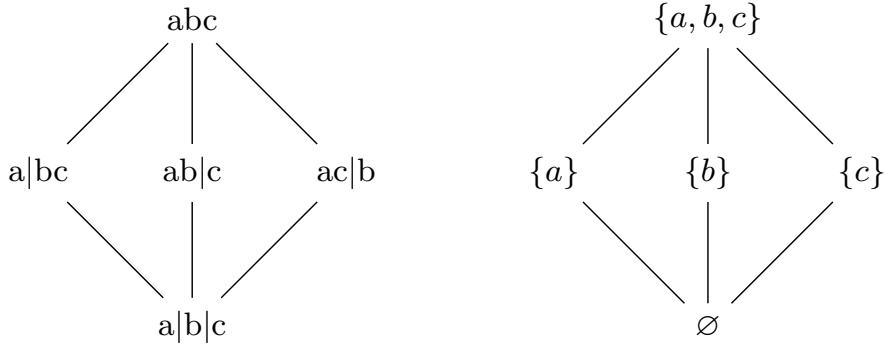
  \node (top) at (0,3) {$\{a,b,c\}$};
  \node (L)   at (-1.5,1.5) {$\{a\}$};
  \node (M)   at (0,1.5) {$\{b\}$};
  \node (R)   at (1.5,1.5) {$\{c\}$};
  \node (bot) at (0,0) {$\varnothing$};

  \draw (top)--(L) (top)--(M) (top)--(R);
  \draw (L)--(bot) (M)--(bot) (R)--(bot);
\end{tikzpicture}%
}
\end{minipage}

\end{tabular}
\end{center}

\subsection{Isomorphisms with subspace lattices}

In this subsection we classify all cases in which a lattice induced by a covering is isomorphic to a subspace lattice.
Let \(V\) be an \(n\)-dimensional vector space over \(\mathbb{F}_q\), and let \(L(V)\) denote the lattice of subspaces of \(V\).

The atoms of $L(V)$ are the \(1\)-dimensional subspaces and
\begin{equation}\label{eq:LV_atoms}
|L^{(1)}(V)|=\binom{n}{1}_q=\frac{q^n-1}{q-1}.
\end{equation}
Moreover, the rank--\(2\) elements are the \(2\)-dimensional subspaces, so
\begin{equation}\label{eq:LV_rank2}
|L^{(2)}(V)|=\binom{n}{2}_q
=\frac{(q^n-1)(q^{n-1}-1)}{(q^2-1)(q-1)}.
\end{equation}
Finally, \(L(V)\) is cover-uniform at the atoms: if \(W\) is a \(1\)-dimensional subspace, then the number of rank--\(2\) elements covering \(W\) equals
\begin{equation}\label{eq:LV_cov}
|\operatorname{Cov}(W)|=\binom{n-1}{1}_q=\frac{q^{n-1}-1}{q-1}.
\end{equation}

\begin{theorem}\label{thm:LV-necessary}
If \(L(\mathcal C)\cong L(V)\), then \(n\leq 2\) (the trivial case).
In the case \(n=2\), we must have
\[
m=q+1
\qquad\text{and}\qquad
k=1.
\]
\end{theorem}

\begin{proof}
Since isomorphisms preserve levels, the number of atoms agrees. Hence, by \eqref{eq:LV_atoms},
\begin{equation}\label{eq:mLV}
m=|L^{(1)}(\mathcal C)|=|L^{(1)}(V)|=\frac{q^n-1}{q-1}.
\end{equation}

\medskip
\noindent\textbf{Case 1: \(\;|L^{(2)}(\mathcal C)|=\binom{m}{2}\).}
Because \(|L^{(2)}(\mathcal C)|=|L^{(2)}(V)|\), using \eqref{eq:mLV} and \eqref{eq:LV_rank2} we obtain
\[
\binom{\frac{q^n-1}{q-1}}{2}
=
\frac{(q^n-1)(q^{n-1}-1)}{(q^2-1)(q-1)}.
\]
Expand the binomial coefficient:
\[
\binom{\frac{q^n-1}{q-1}}{2}
=\frac12\cdot\frac{q^n-1}{q-1}\left(\frac{q^n-1}{q-1}-1\right).
\]
Compute the parenthesis:
\[
\frac{q^n-1}{q-1}-1
=\frac{q^n-1-(q-1)}{q-1}
=\frac{q^n-q}{q-1}
= q\,\frac{q^{n-1}-1}{q-1}.
\]
Therefore,
\[
\binom{\frac{q^n-1}{q-1}}{2}
=\frac12\cdot\frac{q^n-1}{q-1}\cdot q\,\frac{q^{n-1}-1}{q-1}
=\frac{q\,(q^n-1)(q^{n-1}-1)}{2\,(q-1)^2}.
\]
Substituting, we get
\[
\frac{q\,(q^n-1)(q^{n-1}-1)}{2\,(q-1)^2}
=
\frac{(q^n-1)(q^{n-1}-1)}{(q^2-1)(q-1)}.
\]
Assume \(n\ge 2\).
For \(q\ge 2\), we may cancel \((q^n-1)(q^{n-1}-1)\neq 0\), obtaining
\[
\frac{q}{2\,(q-1)^2}=\frac{1}{(q^2-1)(q-1)}.
\]
Cross-multiplying gives
\[
q(q^2-1)(q-1)=2(q-1)^2.
\]
If \(q\neq 1\), cancel \((q-1)\) to obtain
\[
q(q^2-1)=2(q-1)
\iff q(q-1)(q+1)=2(q-1)
\iff q(q+1)=2,
\]
which is impossible for any prime power \(q\ge 2\).
Thus Case 1 yields $n\leq 1$.

\medskip
\noindent\textbf{Case 2: \(\;|L^{(2)}(\mathcal C)|<\binom{m}{2}\).}
In this case we use the cover-uniformity identity (proved earlier):
\begin{equation}\label{eq:kformula}
k=\frac{|L^{(1)}(V)|}{|L^{(1)}(V)|-|\operatorname{Cov}(W)|+1},
\end{equation}
where \(W\) is any atom of \(L(V)\).
Substituting \eqref{eq:LV_atoms} and \eqref{eq:LV_cov} into \eqref{eq:kformula} gives
\[
k
=\frac{\frac{q^n-1}{q-1}}
{\frac{q^n-1}{q-1}-\frac{q^{n-1}-1}{q-1}+1}
=\frac{\frac{q^n-1}{q-1}}
{\frac{q^n-q^{n-1}+q-1}{q-1}}
=\frac{q^n-1}{q^n-q^{n-1}+q-1}.
\]
Factor the denominator:
\[
q^n-q^{n-1}+q-1
=q^{n-1}(q-1)+(q-1)
=(q-1)(q^{n-1}+1),
\]
so
\begin{equation}\label{eq:kfinal}
k=\frac{q^n-1}{(q-1)(q^{n-1}+1)}.
\end{equation}
Since \(k\in\mathbb Z\), we must have \(q^{n-1}+1\mid(q^n-1)\).
But
\[
q^n-1=q(q^{n-1}+1)-(q+1),
\]
so \(q^{n-1}+1\mid(q^n-1)\) implies \(q^{n-1}+1\mid(q+1)\).
If \(n\ge 3\) and \(q\ge 2\), then \(q^{n-1}+1\ge q^2+1>q+1\), a contradiction.
Hence \(n\le 2\).

If \(n=2\), then \eqref{eq:mLV} gives \(m=\frac{q^2-1}{q-1}=q+1\), and \eqref{eq:kfinal} gives
\[
k=\frac{q^2-1}{(q-1)(q+1)}=1.
\]
This completes the proof.
\end{proof}

\begin{theorem}\label{thm:LV-sufficient}
Let \(V\cong \mathbb F_q^n\). If $n=0$ or $n=1$, then the lattice of subspaces is trivially isomorphic to the lattice induced by a covering. Moreover, if \(n=2\), \(m=q+1\), and \(k=1\), then \(L(\mathcal C)\cong L(V)\).
\end{theorem}

\begin{proof}
   All lattices of height  $0$ are isomorphic and all lattices of height  $1$  are isomorphic. Thus, if $n=0$, then $L(\mathcal{C})\cong L(V)$ if and only if $N(L(\mathcal{C}))=1$, which by Lemma~\ref{heights0,1}, is equivalent to $U = \emptyset$. Similarly, if $n=1$, then $L(\mathcal{C})\cong L(V)$ if and only if $\mathcal{C}=\{U\}$.

Now assume \(n=2\).
Then \(L(V)\) is a rank--\(2\) geometric lattice with
\[
|L^{(1)}(V)|=\frac{q^2-1}{q-1}=q+1
\]
atoms. If \(m=q+1\) and \(k=1\), then \(L(\mathcal C)\) is also rank--\(2\) with
\[
|L^{(1)}(\mathcal C)|=m=q+1
\]
atoms. A rank--\(2\) geometric lattice is uniquely determined (up to isomorphism) by its number of atoms, with each atom covering the minimum element and covered by the maximum element.
Hence \(L(\mathcal C)\cong L(V)\).
\end{proof}

Combining Theorems~\ref{thm:LV-necessary} and \ref{thm:LV-sufficient}, we obtain a complete classification of when $L(\mathcal{C})$ can be isomorphic to a lattice of subspaces $L(V)$.

As an example, for \(q=3\) we have \(m=q+1=4\) and \(n=2\).
Taking a covering with \(k=1\) (four singleton blocks) produces \(L(\mathcal C)\cong L(\mathbb F_3^2)\).

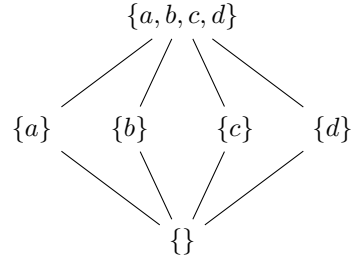
\begin{center}
\begin{tabular}{cc}
    % First Diagram: Lattice of Subspaces of V
    \begin{tikzpicture}
        % Title
        \node at (0,3.8) {\textbf{Lattice of Subspaces of \( V \)}};

        % Nodes
        \node (top) at (0,3) {$\langle (1,0), (0,1) \rangle$};
        \node (11) at (-2,1.5) {$\langle (1,1) \rangle$};
        \node (10) at (-0.7,1.5) {$\langle (1,0) \rangle$};
        \node (01) at (0.7,1.5) {$\langle (0,1) \rangle$};
        \node (21) at (2,1.5) {$\langle (2,1) \rangle$};
        \node (00) at (0,0) {$\langle (0,0) \rangle$};

        % Edges
        \draw (top) -- (11);
        \draw (top) -- (10);
        \draw (top) -- (01);
        \draw (top) -- (21);
        \draw (11) -- (00);
        \draw (10) -- (00);
        \draw (01) -- (00);
        \draw (21) -- (00);
    \end{tikzpicture}
    &
    % Second Diagram: Lattice Induced by Covering
    \begin{tikzpicture}
        % Title
        \node at (0,3.8) {\textbf{Lattice Induced by Covering \( L(\mathcal{C}) \)}};

        % Nodes
        \node (top) at (0,3) {$\{a,b,c,d\}$};
        \node (a) at (-2,1.5) {$\{a\}$};
        \node (b) at (-0.7,1.5) {$\{b\}$};
        \node (c) at (0.7,1.5) {$\{c\}$};
        \node (d) at (2,1.5) {$\{d\}$};
        \node (empty) at (0,0) {$\{\}$};

        % Edges
        \draw (top) -- (a);
        \draw (top) -- (b);
        \draw (top) -- (c);
        \draw (top) -- (d);
        \draw (a) -- (empty);
        \draw (b) -- (empty);
        \draw (c) -- (empty);
        \draw (d) -- (empty);
    \end{tikzpicture}
\end{tabular}
\end{center}

\subsection{Isomorphisms with the Dowling Lattice}

In this subsection we classify all cases in which a lattice induced by a covering is isomorphic to a Dowling lattice.
Let \(G\) be a finite group and let \(Q_n(G)\) denote the Dowling lattice.

It is well known that \(Q_n(G)\) is uniform \citep{Dowling1973}. In particular, we will use the following level counts (for \(n\ge 2\)):
\begin{equation}\label{eq:Q_atoms}
|Q_n^{(1)}(G)| \;=\; n+\binom{n}{2}\,|G|,
\end{equation}
\begin{equation}\label{eq:Q_rank2}
|Q_n^{(2)}(G)| \;=\; \binom{n}{2}+\frac12\,n(n-1)(n-2)\,|G|
+\frac{1}{24}\,n(n-1)(n-2)(3n-5)\,|G|^2,
\end{equation}
and, for any atom \(\pi\in Q_n(G)\),
\begin{equation}\label{eq:Q_cov}
|\operatorname{Cov}(\pi)| \;=\; \binom{n-1}{1}+\binom{n-1}{2}\,|G|.
\end{equation}

\begin{theorem}[Necessity]\label{thm:Qn-necessary}
If \(L(\mathcal C)\cong Q_n(G)\), then either \(n=0,1\) (the trivial cases), or \(n=2\).
In the nontrivial case \(n=2\), we must have
\[
m=|G|+2
\qquad\text{and}\qquad
k=1.
\]
\end{theorem}

\begin{proof}
Since lattice isomorphisms preserve levels, the number of atoms agrees. Hence, by \eqref{eq:Q_atoms},
\begin{equation}\label{eq:mQ}
m=|L^{(1)}(\mathcal C)|=|Q_n^{(1)}(G)|=n+\binom{n}{2}|G|.
\end{equation}

\medskip
\noindent\textbf{Case 1: \(\;|L^{(2)}(\mathcal C)|=\binom{m}{2}\).}
Then \(|Q_n^{(2)}(G)|=|L^{(2)}(\mathcal C)|=\binom{m}{2}\), so
\[
\binom{|Q_n^{(1)}(G)|}{2}-|Q_n^{(2)}(G)|=0.
\]
Using \eqref{eq:Q_atoms} and \eqref{eq:Q_rank2} and expanding gives
\[
\binom{|Q_n^{(1)}(G)|}{2}-|Q_n^{(2)}(G)|
=\frac{1}{12}\,n(n-1)\,|G|\bigl(4n|G|+9-5|G|\bigr).
\]
Thus \(4n|G|+9-5|G|=0\), which has no solution in integers for \(n\ge 2\) and \(|G|\ge 1\).
Hence this case is impossible.

\medskip
\noindent\textbf{Case 2: \(\;|L^{(2)}(\mathcal C)|<\binom{m}{2}\).}
In this case we use the cover-uniformity identity (proved earlier):
\begin{equation}\label{eq:kformulaQ}
k=\frac{|Q_n^{(1)}(G)|}{|Q_n^{(1)}(G)|-|\operatorname{Cov}(\pi)|+1},
\end{equation}
where \(\pi\) is any atom of \(Q_n(G)\).
Substituting \eqref{eq:Q_atoms} and \eqref{eq:Q_cov} into \eqref{eq:kformulaQ} yields
\[
|Q_n^{(1)}(G)|-|\operatorname{Cov}(\pi)|+1
=2+(n-1)|G|,
\qquad
k=\frac{n+\binom{n}{2}|G|}{2+(n-1)|G|}=\frac{n}{2}.
\]
Moreover, by Theorem~\ref{counting} we also have
\begin{equation}\label{eq:Q_thm49}
\begin{aligned}
&\binom{|Q_n^{(1)}(G)|}{2}-|Q_n^{(2)}(G)|
\\
&=\left(\binom{|Q_n^{(1)}(G)|-|\operatorname{Cov}(\pi)|+1}{2}-1\right)
\frac{|Q_n^{(1)}(G)|}{|Q_n^{(1)}(G)|-|\operatorname{Cov}(\pi)|+1}.
\end{aligned}
\end{equation}

Using \( |Q_n^{(1)}(G)|-|\operatorname{Cov}(\pi)|+1 = 2+(n-1)|G|\) and \(\frac{|Q_n^{(1)}(G)|}{|Q_n^{(1)}(G)|-|\operatorname{Cov}(\pi)|+1}=\frac n2\), the right-hand side of \eqref{eq:Q_thm49} becomes
\[
\left(\binom{2+(n-1)|G|}{2}-1\right)\frac n2
=\frac14\,n(n-1)\,|G|\bigl(3+(n-1)|G|\bigr).
\]
Equating this with the expansion from Case 1 gives
\[
\frac{1}{12}\,n(n-1)\,|G|\bigl(4n|G|+9-5|G|\bigr)
=\frac14\,n(n-1)\,|G|\bigl(3+(n-1)|G|\bigr),
\]
which simplifies to \(n|G|-2|G|=0\), hence \(n=2\).
Then \(k=\frac n2=1\), and \eqref{eq:mQ} gives \(m=2+\binom22|G|=|G|+2\).
\end{proof}

\begin{theorem}[Sufficiency]\label{thm:Qn-sufficient}
Let \(G\) be a finite group.
\begin{itemize}
\item If $n=0$ or $n=1$, then the Dowling lattice is trivially isomorphic to the lattice induced by a covering.
\item If \(n=2\), \(m=|G|+2\), and \(k=1\), then \(L(\mathcal C)\cong Q_2(G)\).
\end{itemize}
\end{theorem}

\begin{proof}
 All lattices of height  $0$ are isomorphic and all lattices of height  $1$  are isomorphic. Thus, $L(\mathcal{C})\cong Q_0(G)$ if and only if $N(L(\mathcal{C}))=1$, which by Lemma~\ref{heights0,1}, is equivalent to $U = \emptyset$. Similarly, if $n=1$, then $L(\mathcal{C})\cong Q_1(G)$ if and only if $\mathcal{C}=\{U\}$.

Now assume \(n=2\), \(m=|G|+2\), and \(k=1\).
Then \(Q_2(G)\) is a rank--\(2\) geometric lattice with
\[
|Q_2^{(1)}(G)|=2+|G|
\]
atoms. Under \(k=1\), the lattice \(L(\mathcal C)\) is also rank--\(2\) with
\[
|L^{(1)}(\mathcal C)|=m=|G|+2
\]
atoms. A rank--\(2\) geometric lattice is uniquely determined (up to isomorphism) by its number of atoms, hence \(L(\mathcal C)\cong Q_2(G)\).
\end{proof}

Combining Theorems~\ref{thm:Qn-necessary} and \ref{thm:Qn-sufficient} yields the complete classification of when a covering-induced lattice \(L(\mathcal C)\) can be isomorphic to a Dowling lattice \(Q_n(G)\).

    We proceed to give an example where $n=2, G = \{e,g\}$, and $\mathcal{C}=\{\{a,b,d\},\{a,c,d\}\}$.

    \begin{center}
    \begin{tabular}{cc} 
        % First Diagram: Dowling Lattice
       \begin{tikzpicture}
    % Title
    \node at (0,3.8) {\textbf{Dowling Lattice \( Q_2(G) \), \(G=\{e,g\}\)}};

    % Nodes (top to bottom)
    % Top: both elements in the zero block
    \node (top) at (0,3) {$0\,|\,12$};

    % Rank-1 atoms: two "zero-block" atoms and two "merged-block" atoms
    \node (z1) at (-3,1.5) {$0\,|\,2$};          % 1 in zero block
    \node (m_e) at (-1,1.5) {$12^{\,e}$};        % merge {1,2} with label e
    \node (m_g) at (1,1.5) {$12^{\,g}$};         % merge {1,2} with label g
    \node (z2) at (3,1.5) {$0\,|\,1$};           % 2 in zero block

    % Bottom: discrete partition, zero block empty
    \node (bot) at (0,0) {$1\,|\,2$};

    % Edges
    \draw (top) -- (z1);
    \draw (top) -- (m_e);
    \draw (top) -- (m_g);
    \draw (top) -- (z2);

    \draw (z1) -- (bot);
    \draw (m_e) -- (bot);
    \draw (m_g) -- (bot);
    \draw (z2) -- (bot);
\end{tikzpicture}
        &
        % Second Diagram: Lattice Induced by Covering
        \begin{tikzpicture}
            % Title
            \node at (0,3.8) {\textbf{Lattice Induced by Covering \( L(\mathcal{C}) \)}};

            % Nodes
            \node (top) at (0,3) {$\{a,b,c,d\}$};
            \node (a) at (-2,1.5) {$\{a\}$};
            \node (b) at (-0.7,1.5) {$\{b\}$};
            \node (c) at (0.7,1.5) {$\{c\}$};
            \node (d) at (2,1.5) {$\{d\}$};
            \node (empty) at (0,0) {$\{\}$};

            % Edges
            \draw (top) -- (a);
            \draw (top) -- (b);
            \draw (top) -- (c);
            \draw (top) -- (d);
            \draw (a) -- (empty);
            \draw (b) -- (empty);
            \draw (c) -- (empty);
            \draw (d) -- (empty);
        \end{tikzpicture}
    \end{tabular}
\end{center}

 \bibliographystyle{apalike}
\bibliography{references}
\end{document}